\documentclass[11pt]{amsart}
\usepackage{a4wide}
\usepackage{amssymb}
\usepackage{bbm}

\usepackage{amssymb}

\usepackage{enumerate}
\usepackage[OT1]{fontenc}

\makeatletter
\@namedef{subjclassname@2010}{%

  \textup{2010} Mathematics Subject Classification}

\makeatother

\newtheorem{thm}{Theorem}[section]

\newtheorem{lem}[thm]{Lemma}

\newtheorem{prop}[thm]{Proposition}
\newtheorem{as}{Assumption}

\theoremstyle{definition}

\newtheorem{rem}[thm]{Remark}

\numberwithin{equation}{section}

\newcommand{\ctwo}{\mathcal{C}^{2,1}(\mathbb{R}^{N} \times [0,T)) \cap \mathcal{C}(\mathbb{R}^{N} \times [0,T])}

\newcommand{\id}{i(y,\delta)}
\newcommand{\fd}{f(y,\delta)}
\newcommand{\hd}{h(y,\delta)}

\newcommand{\sid}{i(Y_{t},\delta_{t})}
\newcommand{\sfd}{f(Y_{t},\delta_{t})}

\newcommand{\shd}{h(Y_{k},\delta_{k})}
\newcommand{\sgd}{g(Y_{T})}
\newcommand{\itT}{\int_{t}^{T}}
\newcommand{\iot}{\int_{0}^{t}}

\newcommand{\ex}{\mathbb{E}_{x,y}}
\newcommand{\et}{\mathbb{E}_{x,y,t}}
\newcommand{\ey}{\mathbb{E}_{y,0}}

\begin{document}
\baselineskip=17pt

\title[Discounted reward control problem]{Smooth solutions to discounted reward control problems  with unbounded discount rate and financial applications}

\author[D.Zawisza]{Dariusz Zawisza}

\address{\noindent Dariusz Zawisza, \newline \indent Institute of Mathematics \newline \indent Faculty of Mathematics and Computer Science \newline \indent  Jagiellonian University in Krakow \newline \indent{\L}ojasiewicza  6 \newline \indent 30-348 Krak{\'o}w, Poland }

\email{dariusz.zawisza@im.uj.edu.pl}

\date{}

\begin{abstract}
We consider a discounted reward control problem in continuous time stochastic environment where the discount rate  might be an unbounded function of the control process.
 We provide a set of general assumptions to ensure that there exists
a smooth classical solution to the corresponding HJB equation. Moreover, some  verification reasoning are provided and the possible extension to dynamic games is discussed. At the end of the paper
consumption - investment problems arising in financial economics are considered.
\end{abstract}



\maketitle

\section{Introduction}

We consider a discounted reward stochastic control problem in the continuous time diffusion environment when the control space is a compact set. In addition to the existing literature we assume that the discount rate might be unbounded (the discount rate is denoted further by $h$). We consider both the finite  and the infinite time horizon formulation, with the emphasis on the latter. For both cases we provide a set of assumptions which ensure that the value function of our problem is a smooth classical solution to the corresponding semilinear HJB equation. Such equations are important on its own right, since they arise naturally in many optimization problems, for instance in optimal investment-consumption problems. In such models, the fact that the discount factor might be unbounded give us the possibility to take into account stochastic interest rate models such as the Vasicek model and many others. Aforementioned equations  can be used as well as the first step to solving many unconstrained optimization problems (see Fleming and Hernandez \cite{fleher} moreover,  Friedman \cite{friedman} show that such equations might be employed to solve many deterministic control problems. The proof of main results in the finite horizon case relies on considering first HJB equations with coefficients which satisfy some bounds. Thanks to the stoschastic control representation  of such solution we obtain some estimates on its derivatives and the function itself. Such estimates are further used to apply the Arzell - Ascolli Lemma, when  using bounded functions to approximate unbounded ones.  The proof in the infinite horizon case  is  based on the approximation of the infinite horizon  model using finite horizon control problems and is close to Fleming and McEneaney \cite{mcefle} or  Fleming and Hernandez \cite{fleming}. Apart from the existence theory for HJB equations, we present some results concerning verification arguments and discuss possible extensions to dynamic games. Finally we consider a consumption-investment optimization problem as an example of our general theory.

As it was mentioned our work can be treated as a generalization of some results of Fleming and McEneaney \cite{mcefle}. For the most recent works on the existence of a smooth solution to HJB equations see the work of Rubio \cite{Rubio} (finite horizon problems)  or Lopez-Barrientos et al. \cite{jaso} (infinite horizon formulation). On the other hand there is vast literature about consumption-investment models, where we can find many results concerning existence of solutions to some HJB equations:
 Aktar and Taflin \cite{taflin}, 
Constaneda Leyva and Hernandez \cite{n}, Fleming and Hernandez \cite{fleming}  and Fleming and Pang \cite{fleming2}, Hata and Sheu \cite{hata}, Pang \cite{pang}, Pham \cite{Pham}, Zawisza \cite{zawisza}.
 In most of these works it is usually assumed that the discount rate is constant or bounded from above. As an exception we should mention the work of Fleming and Pang \cite{fleming2} and Pang \cite{pang}, when some specific models with quadratic dependence in the discount rate are considered. Moreover, many of them is focus rather more on some specific unconstrained control problems than on constrained ones.

\section{model}

 Let us consider reference probability system $(\Omega, \mathcal{F}, P)$ with the $N$ dimensonal  Wiener process $\{W_{t}\}_{t \geq 0}$, together with its natural augmented filtration and the stochastic differential equation of the form
\begin{equation} 
dY_{t}=i(Y_{t},\delta_{t})dt + dW_{t}, \quad t \in [0,T] \quad \text{or} \quad t \geq 0\label{ee1}
\end{equation}
$(\delta_{t}, t \in [0,T])$ or respectively $(\delta_{t}, t \geq 0)$, is progressively measurable processes taking values in a fixed compact set $D \subset \mathbb{R}^{k}$. This set of controls  is denoted by  $\mathcal{D}$.  For a notational convenience we write often $Y_{t}$ instead of $Y_{t}^{\delta}$ when there is no possibility of a confusion. We consider here only the trivial diffusion term ($\sigma \equiv 1$), but our results will be valid also for many Lipschitz transformations of $Y$, thus we can manage as well with more complicated dynamics. 
We assume that the controller aim is to maximize
\[
\mathbb{E}_{y,t} \biggl( \int_{t}^{T}e^{\int_{t}^{s}h(Y_{k},\delta_{k})\, dk} f(Y_{s},\delta_{s}) ds  + e^{\int_{t}^{T}h(Y_{k},\delta_{k})\, dk}g(Y_{T})\biggr)
\]
 or in the infinite horizon formulation
 \[
  \mathbb{E}_{y,0} \biggl( \int_{0}^{+\infty}e^{\int_{0}^{s}h(Y_{k},\delta_{k})\, dk} f(Y_{s},\delta_{s}) ds \biggr).
\]
The function $h$ we interpret as the discount rate. The symbol $\mathbb{E}_{y,t}$ is used to reflect the fact that the expected value is calculated under the assumption that the system \eqref{ee1} is starting at time $t$ from the state $y$, but sometimes for a notational convenience we switch the notation and write simply $Y_{k}(y,t)$.

We consider the following set of assumptions.
\begin{as} \label{as} Functions $f$,$g$,$h$,$i$ are continuous, moreover there exist  $L_{1}>0$, $L_{2} \in \mathbb{R} \setminus \{0\}$ such that for $g$, $f$ and $\zeta=f,h,i$ and for all $y, \bar{y} \in \mathbb{R}^{N}$, $\delta \in D$,  we have 
\begin{gather}
|g(y) - g(\bar{y})| \leq L_{1}|y-\bar{y}|, \\ 
|\zeta(y,\delta)-\zeta(\bar{y},\delta)|  \leq  L_1 |y- \bar{y}|, \\
(y-\bar{y})[\id -i(\bar{y},\delta)]  \leq L_{2}|y-\bar{y}|^{2}  \label{lipcond1}.
\end{gather}
\end{as}

Note that under Assumption \ref{as} the unique strong solution to \eqref{ee1} exists.



\begin{prop} \label{kappa}
Under conditions of Assumption \ref{as} for any $T>0$, there exist $M_{T}, K_{T}>0$ that for all $t \in [0,T]$, $y \in \mathbb{R}^{N}$ and for all $\delta \in D$
\begin{equation*} 
\mathbb{E}_{y,0} e^{\int_0^{t} \shd dk}\max \left\{|\sfd|,|g(Y_{t})|,1\right\}  \leq K_{T}e^{M_{T}|y|}.
\end{equation*}
\end{prop} 

\begin{proof}
Note that under Assumption \ref{as} 

$$|Y_{t}| \leq |y| + \max_{t \in [0,T]}|W_{t}| + \iot L_{1}(1+ |Y_{s}|)ds, \quad t \in [0,T], y \in \mathbb{R}^{N}.$$
 
The Gronwall Lemma ensures that 

\begin{equation} \label{inmax}
|Y_{t}|\leq \left(L_{1}T+ |y|+\max_{t \in [0,T]}|W_{t}| \right)e^{L_{1}T},  \quad t \in [0,T], y \in \mathbb{R}^{N}.
\end{equation}

Using the fact that functions  $f, g, h$ are of linear growth condition it is sufficient to observe that
for any $C>0$
\begin{equation} \label{max}
\mathbb{E}e^{C \max_{t \in [0,T]}|W_{t}|} < + \infty.
\end{equation}
 
This easily follows from the fact that for the one dimensional Wiener process we have 

\[
\max_{t \in [0,T]}|W_{t}| \leq \max_{t \in [0,T]}W_{t} + \max_{t \in [0,T]}(-W_{t}) 
\]
and
\[
P(\max_{t \in [0,T]}W_{t} > m) = 2P(W_{T} > m).
\]
\end{proof}
 
 \begin{rem} \label{rem1} In the light of Proposition \ref{kappa} it is worth noticing that there exist a deterministic function $\kappa(t,n)$, $t \in [0,T]$, $n \in \mathbb{N}$,  continuous in $t$ and a sequence $p(T,n)$, $n \in \mathbb{N}$ that for all $\delta \in \mathcal{D} $ we have
 \begin{align}
  \mathbb{E}_{y,0} e^{\int_0^{t} \shd dk}\max \left\{|\sfd|,1\right\} &\leq \kappa(t,n),\label{kappant} \\  
   \mathbb{E}_{y,0} e^{\int_0^{T} \shd dk}\max \left\{|g(Y_{T})|,1\right\} & \leq p(T,n) \quad \text{for all} \quad y \in B(0,n). \label{kappant1}
\end{align}
 \end{rem}

 For infinite horizon control problems conditions \eqref{kappant}-\eqref{kappant1} are to weak. For such problems we need few more conditions, which are given in Assumption \ref{as2} below.
 
\begin{as} \label{as2} There exists a deterministic function $\kappa(t,n)$, $t>0$, $n \in \mathbb{N}$,  continuous in $t$ that for all $\delta \in \mathcal{D}$, $y \in \mathbb{R}$, $t \in [0,T]$, we have
\begin{gather*} 
\mathbb{E}_{y,0} e^{\int_0^{t} \shd dk}\max \left \{|\sfd|,1\right\} \leq \kappa(t,n), \quad \text{for all} \quad y \in B(0,n), \\
\int_{0}^{+\infty} \kappa(t,n) dt < + \infty, \quad \quad \int_{0}^{+\infty} e^{L_{2} t} \kappa(t,n) dt < + \infty,  
\end{gather*}
where the constant $L_{2}$ is taken from Assumption \ref{as} and $B(0,n)$ denotes the closed ball with the radius equals $n$.
\end{as}

\begin{rem}
Conditions of Assumption \ref{as2} are fulfilled for instance if $ \hd=\mu(y,\delta,\eta) - w$, $f$ is bounded, $\mu $ is continuous, bounded and $w>\sup_{y,\delta}\mu(y,\delta)$. Two further  examples might be deduced from results given below. 
\end{rem}

\begin{prop} \label{example} Assume that $N=1$ and the process $(Y_{t},\; t \in [0,T])$ is a strong solution to \eqref{ee1}, suppose further  there exist $\alpha,\beta, P, Q \geq 0, \; \alpha \neq 0$,  that $$\id\leq -\alpha y + \beta, \quad \hd \leq -P + Qy, \quad \delta \in D, \; \eta \in \Gamma, \; y  \in \mathbb{R}.$$
 Then for all $\delta \in \mathcal{D}$
 \[
\mathbb{E}_{y,0} e^{\int_0^{t} \shd dk} \leq e^{\frac{Qy^{+}}{\alpha}} e^{(-P + \frac{Q\beta}{\alpha} +   \frac{Q^{2}}{2\alpha^{2}}) t}, \qquad t>0, \;  y  \in \mathbb{R}.
\]
\end{prop}
\begin{proof}
Note that 
$$dY_{t}=i(Y_{t},\delta_{t})dt +  dW_{t}$$
Using the Ito formula we have
\[
de^{\alpha t} Y_{t}= [\alpha e^{\alpha t} Y_{t}+ e^{\alpha t} \sid]dt+ e^{\alpha t} dW_{t} \leq   \beta e^{\alpha t} + e^{\alpha t} dW_{t},
\]
and further
\[
e^{\alpha t} Y_{t} = y + \int_{0}^{t} [\alpha e^{\alpha s} Y_{s}+ e^{\alpha s} i(Y_{s},\delta_{s}) ] ds + \int_{0}^{t} e^{\alpha s} dW_{s} \leq y  +  \int_{0}^{t}\beta e^{\alpha s} ds + \int_{0}^{t}e^{\alpha s} dW_{s}.
\]
This yields
\[
Y_{t} \leq ye^{-\alpha t} + \beta \frac{1}{\alpha}(1-e^{-\alpha t}) + \int_{0}^{t} e^{\alpha(s - t)} dW_{s}
\]
and 
\begin{multline*}
 \int_{0}^{t} Y_{s} \leq \frac{y^{+}}{\alpha}  + \frac{\beta}{\alpha} t + \int_{0}^{t} \int_{0}^{s} e^{\alpha(k - t)} dW_{k} ds =\frac{y^{+}}{\alpha}  + \frac{\beta}{\alpha} t \\+ \int_{0}^{t} \int_{s}^{t} e^{\alpha(k - t) } dk dW_{s}= \frac{y^{+}}{\alpha}  + \frac{\beta}{\alpha} t + \int_{0}^{t} \frac{1}{\alpha} \left(1-e^{\alpha(s-t)}\right) dW_{s}.
\end{multline*}
Therefore
\[
 \int_{0}^{t}\shd dk \leq -Pt + \frac{Qy^{+}}{\alpha}  + \frac{Q\beta}{\alpha} t +  Q \int_{0}^{t} \frac{1}{\alpha} \left(1-e^{\alpha(s-t)}\right) dW_{s}
\]
and consequently
\[
\mathbb{E}_{y,0} e^{\int_0^{t} h(Y_{k},\delta_{k}) dk}f(Y_{t},\delta_{t}) \leq e^{\frac{Qy^{+}}{\alpha}} e^{(-P + \frac{Q\beta}{\alpha} +   \frac{Q^{2}}{2\alpha^{2}}) t}.
\]
\end{proof}

It is not very hard to extend the above result to the case $y \in \mathbb{R}^{N}$. The second example is presented in the following proposition.

\begin{prop}
 If $\hd \leq-w$, $\fd \leq L_{1}(1+|y|)$ where $w>\max\{0,L_{2}\}$, then  for all  $\delta \in \mathcal{D}$
 \[
\mathbb{E}_{y,0} e^{\int_0^{t} h(Y_{k},\delta_{k}) dk}\shd\leq e^{-wt}(1+|y|e^{L_{2}t}),  \qquad y \in \mathbb{R}^{N}.
\]
\end{prop}
\begin{proof}
 Using the Ito formula we have 
$$ \mathbb{E}_{y,0} |Y_{k}|^{2} \leq |y|^{2} + 2L_{2} \int_{0}^{k} \mathbb{E}_{y,0} |Y_{l}|^{2} dl.
$$
Gronwall's lemma yields 
$$
\mathbb{E}_{y,0} |Y_{k}|^{2} \leq |y|^{2} e^{2L_{2}(k-s)}
$$
and finally
$$
\mathbb{E}_{y,0} |Y_{k}| \leq |y| e^{L_{2}k}.
$$
This  completes the proof.
\end{proof}
\section{Finite horizon problem}
We start with the HJB equation of the form
\begin{equation} \label{equationL}
 u_{t} + \frac{1}{2} \Delta u +\max_{\delta \in D} \biggl(i(y,\delta) \nabla u  +  \hd u+ \fd\biggr)=0, \quad y \in \mathbb{R}^{N},  t \in [0,T)
\end{equation}
with the terminal condition $u(y,T)=g(y)$. The symbol $i(y,\delta) \nabla u$ is used to note the dot product between the function $i$ and gradient of $u$.
It is already well known that under some mild conditions there exists a smooth solution to that equation, so we consider first the following form of the verification theorem.
\begin{prop} \label{lem1}
 Suppose that all conditions of Assumption \ref{as}  are  satisfied and  there exists $u$ - a  solution to  \eqref{equationL} together with $K_{T}, M_{T} >0$, such that $u \in \mathcal{C}^{2,1}(\mathbb{R}^{N} \times [0,T)) \cap \mathcal{C}(\mathbb{R}^{N} \times [0,T])$ and
 \begin{equation} \label{2ee}
 |u(y,t)|\leq K_{T} e^{M_{T}|y|}, \quad y \in \mathbb{R}^{N}, t \in [0,T].
 \end{equation}
  Then $u$  admits a stochastic representation of the form
\begin{align} \label{repstoch}
u(y,t)=  \sup_{\delta \in \mathcal{D}}  \mathbb{E}_{y,t} \biggl( \int_{t}^{T}e^{\int_{t}^{s}h(Y_{k},\delta_{k})\, dk} f(Y_{s},\delta_{s}) ds  + e^{\int_{t}^{T}h(Y_{k},\delta_{k})\, dk}g(Y_{T})\biggr),
\end{align} 
where $Y$ is the unique solution  to $dY_{t}=i(Y_{t},\delta_{t})dt +  dW_{t}$. Moreover, if $\delta^{*}(y,t)$ is a Borel measurable maximizer in equation \eqref{equationL} then it determines an optimal control process . 
\end{prop}
\begin{proof} 
We can combine \eqref{2ee}, \eqref{inmax}, \eqref{max} to ensure that for any $y \in \mathbb{R}^{N}$ and $\delta \in \mathcal{D}$, we have
\[
  \mathbb{E}_{y,t} \sup_{0 \leq s \leq T}e^{\int_{t}^{s}h(Y_{k},\delta_{k})\, dk} |u(Y_{s}^{\delta},s)| < +\infty. 
\]
Suppose further $\delta^{*}(y,t)$ is a Borel measurable maximizer of \eqref{equationL}. By the result of Veretennikov \cite{veretennikov} extended in McEneaney \cite[Lemma 3.2.1]{mce} there exists a strong solution to 
\[
dY_{t}=i(Y_{t},\delta^{*}(Y_{t}))dt + dW_{t}.
\]
In fact the result of McEneaney was proved under the assumption that the cofficient $L_{2}$ from our Assumption \ref{as} is strictly less than 0, but his crucial inequality 2.7 holds also if the function $i$ is Lipschitz continuous. 
These  facts enable us to use standard verification reasoning, which can be found in many classical textbooks.
\end{proof}

Let $u^{T}(y,t)$ denote the solution to \eqref{equationL} with terminal condition given at time T.

Now we are mainly interested in proving some estimates for $u^{T}(y,0)$ and its first $y$ - derivatives, which will be further used to apply the Arzell - Ascolli Lemma when passing to the suitable limit. Note that representation \eqref{repstoch} guarantees uniqueness in the class of functions satisfying \eqref{2ee} and this may lead us to the equality $u^{T-t}(y,0)=u^{T}(y,t)$ (assuming that both solutions exists).


We will need the following lemma
\begin{lem} \label{basic} Suppose that all conditions of Assumption 1 are satisfied,  then 
\begin{equation*} \label{basicinequality}
 |Y_{k}(y,s) - Y_{k}(\bar{y},s)|  \leq |y-\bar{y}| e^{L_{2}(k-s)}, \quad \text{ for all} \quad y, \bar{y} \in \mathbb{R}^{N}, k \geq s \geq 0.
\end{equation*}
\end{lem}

\begin{proof}
Let $k \geq s$ be fixed. Using the Ito formula we have 
\begin{multline*}
 |Y_{k}(y,t) - Y_{k}(\bar{y},t)|^{2} \\= (y - \bar{y})^{2} +  \int_{t}^{k} 2(Y_{l}(y,t) - Y_{l}(\bar{y},t))[i(Y_{l}(y,t),\delta_{l})-i(Y_{l}(y,t),\delta_{l})] \, dl. 
\end{multline*}
Using \eqref{lipcond1} we get
$$ |Y_{k}(y,t) - Y_{k}(\bar{y},t)|^{2} \leq |y - \bar{y}|^{2} + 2L_{2} \int_{t}^{k}  |Y_{l}(y,t) - Y_{l}(\bar{y},t)|^{2} dl.
$$
Gronwall's lemma yields 
$$
 |Y_{k}(y,s) - Y_{k}(\bar{y},s)|^{2} \leq |y-\bar{y}|^{2} e^{2L_{2}(k-s)}.
$$
\end{proof}

\begin{prop} \label{lembounds}
 Suppose that all conditions of Assumption \ref{as} and  there exists $u^{T}$ - a bounded solution to  \eqref{equationL}. Moreover let the function $f$ be bounded and $h$ be bounded from above.  Then  for all $n \in \mathbb{N}$ the following estimates for $u^{T}$  are satisfied:
\begin{align}  
\left|u^{T}(y,0)\right| &\leq    \int_{0}^{T} \kappa(s,n)  ds +p(T,n), \notag  \\
\left|\nabla u^{T}(y,0)\right| & \leq  \left(L_{1}+\frac{L_{1}}{|L_{2}|}\right) \biggl( \int_{0}^{T} \max \{ 1, e^{L_{2}s}\} \kappa(s,n)  ds + \max \{ 1, e^{L_{2}T}\} p(T,n)  \biggr), \notag
\end{align}
where the function $\kappa(t,n)$ and $p(T,n)$ are taken from \eqref{kappant}-\eqref{kappant1}.
\end{prop}

\begin{proof}
Let's fix $y \in B(0,n)$
Proposition \ref{lem1} ensures that 
$$u^{T}(y,0)= \sup_{\delta \in \mathcal{D}
} \mathbb{E}_{y,0} \biggl( \int_{0}^{T}e^{\int_{0}^{s}h(Y_{k},\delta_{k})\, dk} f(Y_{s},\delta_{s})  ds +e^{\int_{0}^{T}h(Y_{k},\delta_{k})\, dk}\sgd \biggr).
$$ 

From Remark \ref{rem1} we get
\begin{align*}
|u^{T}(y,0)| &\leq \int_{0}^{T} \kappa(s,n)ds + p(T,n). 
\end{align*} 

The bound for $\nabla u$  will be obtained by estimating the Lipschitz constant.   For a notational convenience we will write $\mathbb{E}f(Y_{t}(y,s))$ instead of  $\mathbb{E}_{y,s}f(Y_{t})$.

\begin{align*}
 |&u^{T}(y,0)-u^{T}(\bar{y},0)|  \\
\quad &\begin{aligned}[t]&\leq  \begin{aligned}[t]  \sup_{\delta \in \mathcal{D}}  \mathbb{E} \int_{0}^{T} \left|f(Y_{s}(\bar{y},0),\delta_{s})  \right|  \biggl| e^{\int_{0}^{s} h(Y_{k}(y,0),\delta_{k})\, dk} -   e^{\int_{0}^{s} h(Y_{k}(\bar{y},t),\delta_{k})\, dk} \biggr| ds
\end{aligned} \\
& \quad +  \sup_{\delta \in \mathcal{D}}  \mathbb{E} \int_{0}^{T} e^{\int_{0}^{s} h(Y_{k}(y,0),\delta_{k})\, dk}\biggl|f(Y_{s}(y,0),\delta_{s}) -f(Y_{s}(\bar{y},0),\delta_s)  \biggr| ds\\
& \quad +  \sup_{\delta \in \mathcal{D}}  \mathbb{E} \left|g(Y_{T}(\bar{y},0))  \right|\biggl| e^{\int_{0}^{T} h(Y_{k}(y,0),\delta_{k})\, dk} -   e^{\int_{0}^{T} h(Y_{k}(\bar{y},t),\delta_{k})\, dk} \biggr| \\
& \quad +  \sup_{\delta \in \mathcal{D}}  \mathbb{E}  e^{\int_{0}^{T} h(Y_{k}(y,0),\delta_{k})\, dk}\biggl|g(Y_{T}(\bar{y},0)) -g(Y_{T}(y,0))  \biggr|.
\end{aligned}
\end{align*}

Note that Assumption 1 is satisfied, hence 
\begin{align*}
 \int_{0}^{s} [h(Y_{k}(y,t),\delta_{k})-h(Y_{k}(\bar{y},0),\delta_{k})]\, dk \leq L_{1} &\int_{0}^{s} |Y_{k}(y,0) - Y_{k}(\bar{y},0)| dk \\ &\leq  L_{1} |y - \bar{y}|\int_{0}^{s} e^{L_{2} k} dk 
\end{align*} 

and consequently
\begin{align*}
\int_{0}^{s} h(Y_{k}(y,t),\delta_{k},\eta(\delta_{k})) dk &\leq \int_{0}^{s} h(Y_{k}(\bar{y},0),\delta_{k},\eta(\delta_{k})) dk + L_{1} |y - \bar{y}|\int_{0}^{s} e^{L_{2} k} dk. 
\end{align*}

Which means that 
\begin{multline*}
 \left|f(Y_{s}(\bar{y},0),\delta_{s})  \right| \biggl| e^{\int_{0}^{s} h(Y_{k}(y,0),\delta_{k})\, dk} -   e^{\int_{0}^{s} h(Y_{k}(\bar{y},0),\delta_{k})\, dk} \biggr| \leq \\ \left|f(Y_{s}(\bar{y},0),\delta_{s})\right|e^{\int_{0}^{s} h(Y_{k}(\bar{y},t),\delta_{k}) dk + L_{1} |y - \bar{y}|\int_{0}^{s} e^{L_{2} k} dk} \cdot \\  \cdot  \int_{0}^{s} |h(Y_{k}(y,t),\delta_{k})-h(Y_{k}(\bar{y},t),\delta_{k})|\, dk \\ \leq L_{1} |y-\bar{y}| \left|f(Y_{s}(\bar{y},0),\delta_{s})  \right|e^{\int_{0}^{s} h(Y_{k}(\bar{y},t),\delta_{k}) dk + L_{1} |y - \bar{y}|\int_{0}^{s} e^{L_{2} k} dk}  \int_{0}^{s} e^{L_{2} k} dk \\
 \leq \frac{L_{1}}{|L_{2}|} \max \{ 1, e^{L_{2}s}\}|y-\bar{y}| \left|f(Y_{s}(\bar{y},0),\delta_{s})  \right|e^{\int_{0}^{s} h(Y_{k}(\bar{y},t),\delta_{k}) dk + L_{1} |y - \bar{y}|\int_{0}^{s} e^{L_{2} k} dk}, 
\end{multline*}
where in the first inequality we use the fact that for any $x,y\leq a$, $|e^x-e^y| \leq e^{a} |x-y|$. 
The same reasoning might be used  to obtain 

\begin{multline*}
\left |g(Y_{T}(\bar{y},0)) \right| \biggl| e^{\int_{0}^{T} h(Y_{k}(y,0),\delta_{k})\, dk} -   e^{\int_{0}^{T} h(Y_{k}(\bar{y},0),\delta_{k})\, dk} \biggr| \\ \leq \frac{L_{1}}{|L_{2}|} \max \{ 1, e^{L_{2}T}\}|y-\bar{y}| \left| g(Y_{T}(\bar{y},0))\right|  |e^{\int_{0}^{T} h(Y_{k}(\bar{y},t),\delta_{k}) dk + L_{1} |y - \bar{y}|\int_{0}^{T} e^{L_{2} k} dk},
\end{multline*}

\begin{equation*}
 e^{\int_{0}^{T} h(Y_{k}(y,0),\delta_{k})\, dk}\left|g(Y_{T}(y,0)) -g(Y_{T}(\bar{y},0))  \right| ds \leq L_{1}   |y - \bar{y}|   e^{L_{2}T}  e^{\int_{0}^{T} h(Y_{k}(y,0),\delta_{k})\, dk} ,
\end{equation*}

\begin{equation*}
 e^{\int_{0}^{s} h(Y_{k}(y,0),\delta_{k})\, dk}\biggl|f(Y_{s}(y,0),\delta_{s}) -f(Y_{s}(\bar{y},0),\delta_{s})  \biggr| ds 
\leq L_{1}   |y - \bar{y}|   e^{L_{2}s} e^{\int_{0}^{s} h(Y_{k}(y,0),\delta_{k})\, dk} .
\end{equation*}

Now we can use Assumption \ref{as2} to summarize all inequalities into
\begin{multline*}
|u^{T}(y,0)-u^{T}(\bar{y},0)| \\ 
\leq  \left(L_{1}+\frac{L_{1}}{|L_{2}|}\right)|y-\bar{y}|\mathbb{E}\int_{0}^{T}\max \{ 1, e^{L_{2}s}\}\left|f(Y_{s}(\bar{y},0),\delta_{s})  \right| e^{\int_{0}^{s} h(Y_{k}(\bar{y},t),\delta_{k},\eta(\delta_{k})) ds + L_{1} |y - \bar{y}|\int_{0}^{s} e^{L_{2} k} dk} ds \\
 \leq  \left(L_{1}+\frac{L_{1}}{|L_{2}|}\right)   |y-\bar{y}|\int_{0}^{T} \max \{ 1, e^{L_{2}s}\}  e^{ L_{1} |y - \bar{y}|\int_{0}^{s} e^{L_{2} k} dk}\kappa(s,n)ds  \\
 + \max \{ 1, e^{L_{2}T}\}  e^{ L_{1} |y - \bar{y}|\int_{0}^{T} e^{L_{2} k} dk}p(T,n),
\end{multline*}

which guarantees that 
\[
\left|\nabla u^{T}(y,0)\right|  \leq  \left(L_{1}+\frac{L_{1}}{|L_{2}|}\right) \biggl( \int_{0}^{T} \max \{ 1, e^{L_{2}s}\} \kappa(s,n)  ds + \max \{ 1, e^{L_{2}T}\} p(T,n)  \biggr).
\]

\end{proof}

%

Now we are ready to prove that the value function is a smooth solution to  \eqref{equationL} also for unbounded functions $f$ and $h$.

\begin{lem} \label{tofriedman}
 Suppose that all conditions of  Assumption \ref{as} are satisfied and let the function $f$ be bounded and $h$ be bounded from above. Then the function \\ $H(y,u,p)=\max_{\delta \in D} \left(\id p  +  \hd u + \fd\right) $ is Lipschitz continuous on compact subsets on $\mathbb{R}^{N} \times  \mathbb{R} \times  \mathbb{R}^{N}$   and  there exists $K>0$ that 
\begin{align}
&|H(y,0,0)| \leq K, \qquad y \in \mathbb{R}^{N} \notag\\
&H(y,u,p)-H(y,\bar{u},p) \leq K \left( u-\bar{u}\right), \qquad \text{if} \quad u > \bar{u}, \; y \in  \mathbb{R}^{N}, p \in  \mathbb{R}^{N} \label{uu} \\
&|H(y,u,p)-H(\bar{y},u,p)| \leq K (1+|p|)|y-\bar{y}|,  \qquad u \in \mathbb{R}, \; y,\bar{y}, p \in  \mathbb{R}^{N} \notag \\
&|H(y,u,p)-H(y,u,\bar{p})| \leq K (1+|y|)|p-\bar{p}| \qquad u \in \mathbb{R}, \; y, p, \bar{p} \in  \mathbb{R}^{N} .  \notag
\end{align}

\end{lem}

\begin{proof}
Let's define
\[
G(\delta,y,u,p)=\id p  +  \hd u + \fd.
\]
We start with the proof of inequality \eqref{uu}. Observe that we have
\[
\max_{\delta \in D} G(y,u,p,\delta)-\max_{\delta \in D} G(y,\bar{u},p,\delta) \leq  \max_{\delta \in D}  (G(y,u,p,\delta)-G(y,\bar{u},p,\delta)).
\]
and
\[
G(y,u,p,\delta)-G(y,\bar{u},p,\delta) = \hd(u - \bar{u}) \leq h^{+}(y,\delta)(u - \bar{u}), 
\]
which gives us the desired result.

For the rest of the proof it is sufficient to note that $G$ is continuous and
$$ |\max_{\delta \in D} G(y,u,p,\delta)-\max_{\delta \in D} G(\bar{y},\bar{u},\bar{p},\delta)| \leq  \max_{\delta \in D}  | G(y,u,p,\delta)-G(\bar{y},\bar{u},\bar{p},\delta)|.$$

\end{proof}

\begin{prop} \label{unbounded}
Suppose that all conditions of Assumption \ref{as}  are  satisfied. Then 
 \begin{equation*} 
u(y,t)=  \sup_{\delta \in \mathcal{D}}  \mathbb{E}_{y,t} \biggl( \int_{t}^{T}e^{\int_{0}^{s}h(Y_{k},\delta_{k})\, dk} f(Y_{s},\delta_{s}) ds \biggr)
\end{equation*} 
is a smooth ($\mathcal{C}^{2,1}(\mathbb{R}^{N}\times [0,T)) \cap \mathcal{C}(\mathbb{R}^{N}\times [0,T]) $) solution to 
\[ 
\begin{cases}
u_{t}+\frac{1}{2} \Delta u + \max_{\delta \in D} \biggl(\id \nabla u  +  \hd u+ \fd\biggr)=0, \quad  t \in [0,T), \; y \in \mathbb{R}^{N} \\
u(y,T)=g(y).
\end{cases}
\]
Moreover for all $t \in [0,T]$ and $y \in \mathbb{R}^{N}$, we have
\begin{align*}  
\left|u(y,t)\right| &\leq    \int_{0}^{T-t} \kappa(s,n)  ds +p(T-t,n), \\
\left|\nabla u(y,t)\right| & \leq  \left(L_{1}+\frac{L_{1}}{|L_{2}|}\right) \biggl( \int_{0}^{T-t} \max \{ 1, e^{L_{2}s}\} \kappa(s,n)  ds + \max \{ 1, e^{L_{2}(T-t)}\} p(T-t,n)  \biggr).
\end{align*}
\end{prop}

\begin{proof}
Let's define the following sequences of functions:
$$
h_{k}(y,\delta)=\begin{cases}
\hd \quad  & \text{if} \quad |y| \leq k,\\
 h^{+}(y,\delta)\left(2-\frac{|y|}{k}\right)-h^{-}(y,\delta), & \text{if} \quad k \leq |y| \leq 2k,\\
  -h^{-}(y,\delta)       &\text{if} \quad  |y| \geq 2k,  \end{cases}
$$ 

$$
f_{k}(y,\delta)=\begin{cases}
\fd \quad  & \text{if} \quad |y| \leq k,\\
 \fd\left(2-\frac{|y|}{k}\right), & \text{if} \quad k \leq |y| \leq 2k,\\
  0        &\text{if} \quad  |y| \geq 2k,  \end{cases}
$$ 

$$
g_{k}(y)=\begin{cases}
g(y) \quad  & \text{if} \quad |y| \leq k,\\
 g(y)\left(2-\frac{|y|}{k}\right), & \text{if} \quad k \leq |y| \leq 2k,\\
  0        &\text{if} \quad  |y| \geq 2k.  \end{cases}
$$

Note that 
\begin{gather*} 
\lim_{k \to \infty}  h_{k}(y,\delta)= \hd, \quad \lim_{k \to \infty}  f_{k}(y,\delta)= \fd, \quad \lim_{k \to \infty}  g_{k}(y)= g(y), \\
h_{k}(y,\delta) \leq \hd, \quad |f_{k}(y,\delta)| \leq |\fd|, |g_{k}(y)| \leq |g(y)| \\
|h_{k}(y,\delta,\eta)-h_{k}(\bar{y},\delta)| \leq L_{1}^{k} |y-\bar{y}| , \quad |f_{k}(y,\delta)-f_{k}(\bar{y},\delta)| \leq L_{1}^{k} |y-\bar{y}|, \\
|g_{k}(y)-g_{k}(\bar{y})| \leq L_{1}^{k} |y-\bar{y}|,
\end{gather*}
where $L_{1}^{k}:=2L_{1}(1+\frac{1}{k})$.
In addition $f_{k}, g_k$ are bounded and $h_{k}$ is bounded from above.
 Therefore, gathering Lemma \ref{tofriedman} and   Friedman \cite[Theorem 2.1]{friedman} we get that for all $T>0$ there exists $u_k$ - a bounded solution to  the Cauchy problem
\begin{equation} \label{eqseq} 
 u_{t}+ \frac{1}{2} \Delta u  +\max_{\delta \in D} \biggl(\id \nabla u  +  h_{k}(y,\delta)u + f_{k}(y,\delta)\biggr)  =0, 
\end{equation}
with terminal condition $u(y,T)=g_{k}(y)$.

Proposition \ref{lembounds} ensures that the solution to \eqref{eqseq} has a representation	
\begin{equation*} 
u_{k}(y,t)=  \sup_{\delta \in \mathcal{D}}   \mathbb{E}_{y,t} \biggl( \int_{t}^{T}e^{\int_{t}^{s}h_{k}(Y_{l},\delta_{l})\, dl} f_{k}(Y_{s},\delta_{s})  ds \biggr),
\end{equation*} 
and for all $y \in B(0,n)$ the following inequalities are satisfied:

\begin{align}  
\left|u_{k}(y,t)\right| &\leq  \int_{0}^{T-t} \kappa(s,n)  ds +p(T-t,n), \notag \\
\left|\nabla u_{k}(y,t)\right| & \leq \left(L_{1}^{k}  + \frac{L_{1}^{k}}{|L_{2}|}\right)   \biggl( \int_{0}^{T-t} \max \{ 1, e^{L_{2}s}\} \kappa(s,n)  ds + \max \{ 1, e^{L_{2}(T-t)}\} p(T-t,n)  \biggr),\notag
\end{align}

To obtain local uniform bounds for other derivatives and its Lipschitz constants we can multiply $u_{k}$ by the function $\alpha $ of the class $\mathcal{C}^{2}$ with compact support such that $\alpha \equiv 1$ on the set $B(0,n)\times (t_{0},t_{1})$. Now we can combine (E8) and (E9) from Fleming and Rischel \cite{FlemingRischel} to obtain desired uniform bound on $B(0,n) \times (t_{0},t_{1}) $ for first derivatives of $u_{k}$. Bounds for second derivatives  $\partial_{i} \partial_{j} u_{k}$ we can obtain by standard Schauder estimates.  By applying the standard argument with a help of the Arzel Ascolli Lemma we can deduce that there exits $\tilde{u}(y,t)$ the limit of some subsequence of $u_{k}$. What is more, the convergence holds locally uniformly in $(y,t)$ together with all suitable derivatives. Thanks to that we immediately get that $\tilde{u} \in \ctwo$ and is a solution to 
\[ 
\begin{cases}
\tilde{u}_{t}+\frac{1}{2} \Delta \tilde{u} + \max_{\delta \in D} \biggl(\id \nabla \tilde{u}  +  \hd \tilde{u}+ \fd\biggr)=0, \quad  t \in [0,T), \; y \in \mathbb{R}^{N} \\
\tilde{u}(y,T)=g(y).
\end{cases}
\]
Since the condition \ref{kappant} is fulfilled we can apply Proposition \ref{lem1} to obtain
\[
\tilde{u}(y,t)=u(y,t)= \sup_{\delta \in \mathcal{D}}  \mathbb{E}_{y,t} \biggl( \int_{t}^{T}e^{\int_{t}^{s}h(Y_{k},\delta_{k})\, dk}  f(Y_{s},\delta_{s}) ds  + e^{\int_{t}^{T}h(Y_{l},\delta_{l})\, dl} g(Y_{T})\biggr).
\]
\end{proof}

\section{Infinite horizon analogue}
Once we have the desired result for finite horizon problems we can pass the time horizon to infinity and prove suitable results for infinite horizon problems.
\begin{thm} \label{main} Suppose that all conditions of Assumption \ref{as} and  Assumption \ref{as2} are satisfied. Then 
\begin{equation*} 
v(y)=  \sup_{\delta \in \mathcal{D}}  \mathbb{E}_{y,0} \biggl( \int_{0}^{+\infty}e^{\int_{0}^{s}h(Y_{k},\delta_{k})\, dk} f(Y_{s},\delta_{s}) ds \biggr)
\end{equation*} 
is a smooth ($\mathcal{C}^{2}(\mathbb{R}^{N})$) solution  to 
\begin{equation} \label{infinitequation}
 \frac{1}{2}\Delta v +\max_{\delta \in D} \biggl(\id \nabla v  +  \hd v + \fd\biggr)=0.  
\end{equation}
 
\end{thm}
\begin{proof} 
  Let $u^{T}$ be the solution to 
\begin{equation*} 
 u_{t}+ \frac{1}{2} \Delta u   +\max_{\delta \in D} \biggl(\id\nabla u  +  \hd u + \fd\biggr) =0, 
\end{equation*}
constructed in Proposition \ref{unbounded}.
It is important to notice here that 
 $v(y,t)=u^{T}(y,T-t)=u^{t}(y,0)$ is a solution to 
\[ 
\begin{cases}
v_{t}-\frac{1}{2} \Delta v- \max_{\delta \in D} \biggl(\id \nabla v  +  \hd v+ \fd\biggr)=0, \quad  t \in (0,T] \\
v(y,0)=0.
\end{cases}
\]
and for all $y \in B(0,n)$ 

\begin{align}  
\left|v(y,t)\right| &\leq L_{1}  \biggl( \int_{0}^{t} \kappa(s,n)  ds \biggr) , \notag \\
\left|\nabla v(y,t)\right| & \leq \left(L_{1}  + \frac{L_{1}}{|L_{2}|}\right)  \biggl( \int_{0}^{t}  \max \{1,e^{L_{2} s}\}\kappa(s,n) ds \biggr),\notag
\end{align}

Now let us consider the estimate on  $\frac{\partial }{\partial t }v(y,t)$. Namely, let $t > 0$ be fixed. Observe that for $0<\xi<t$ 
\begin{align*}
|u^{t}(y,0)-u^{t-\xi}(y,0)| \leq  \sup_{\delta \in \mathcal{D}}  \biggl|I(t,y,\eta,\delta) -I(t-\xi,y,\eta,\delta) \biggr|,
\end{align*}
where 
$$I(t,y,\eta,\delta):= \mathbb{E}_{y,0}  \biggl( \int_{0}^{t}e^{\int_{0}^{s}h(Y_{k},\delta_{k})\, dk} |f(Y_{s},\delta_{s})| ds \biggr). $$
Note that 
\begin{align*}
\biggl| I(t,y,\delta) -I(t-\xi,y,\delta) \biggr| &= \mathbb{E}_{y,0}\int_{t-\xi}^{t} e^{\int_{0}^{s}h(Y_{k},\delta_{k})\, dk}| f(Y_{s},\delta_{s})| ds \\ & \quad \leq  \int_{t-\xi}^{t} \kappa(s,n) ds \leq \xi \max_{s \in [t-\xi,t]} \kappa(s,n),  
\end{align*}
which yields 
\begin{equation}
\left|\frac{\partial }{\partial t }v(y,t)\right| \leq  \kappa(t,n)  \label{bounddert}.
\end{equation}

The solution  will be constructed by taking the limit $v(y)=\lim_{t \to \infty} v(y,t)$ (such limit exists since Assumption \ref{as2} is satisfied).
As in the proof of the previous theorem, we can combine E8 and E9 from Fleming and Rishel \cite{FlemingRischel}
to obtain suitable bounds for all derivatives.
By the Arzel-Ascolli Lemma, for each $B(0,n)$ there exist a sequence $(t_{n}, n=1,2,\ldots)$ such that $v(y,t_{n})$ is convergent to some twice continuously differentiable function, what is more, the convergence holds locally uniformly together with suitable derivatives. Moreover from  \eqref{bounddert} it follows that $\lim_{n \to \infty}\frac{\partial }{\partial t } v(y,t_{n})=0$. This indicates that $v$ is a solution to \eqref{infinitequation}.
\end{proof}

It should be noticing here that in the infinite horizon case we have proved only that the value function is a smooth solution to PDE, but still we are not sure if the maximizer in \eqref{infinitequation} determines the optimal control for our problem.  Instead of standard reasoning we propose to use the following result, which can be applied for many classical problem formulations.
\begin{prop} \label{convergence}
Suppose that $\lim_{n \to \infty} T_{n} = +\infty$ and $(\delta^{n}, \;n \in \mathbb{N})$ is a sequence of progressively measurable processes such that $\lim_{n \to \infty} \delta_{t}^{n} = \delta_{t}$ a.s. and $\lim_{n \to \infty}Y_{t}^{\delta^n} \to Y_{t}^{\delta}$ a.s.,  then under conditions of Theorem \ref{main} we have
\[
\lim_{n \to \infty} \ey \int_{0}^{T_{n}}e^{\int_{0}^{s}h(Y_{k}^{\delta^{n}},\delta_{k}^{n})\, dk} |f(Y_{s}^{n},\delta_{s}^{n})| ds =\ey \int_{0}^{+\infty}e^{\int_{0}^{s}h(Y_{k}^{\delta},\delta_{k})\, dk} |f(Y_{s},\delta_{s})| ds.
\]
 
\end{prop} 
\begin{proof}
Fix $\varepsilon >0$. Note that there exists $n_{0} \in \mathbb{N}$ such that for any $n \geq n_{0}$
\[
\sup_{\delta \in \mathcal{D}} \ey \int_{T_{n}}^{+\infty}e^{\int_{0}^{s}h(Y_{k},\delta_{k}^{n})\, dk} |f(Y_{s},\delta_{s})| ds < \frac{\varepsilon}{3}
\]
 and there exists $n'_{0}$ such that for any $n \geq n_{0}'$
\[
\left|\ey \int_{0}^{T_{n_{0}}}e^{\int_{0}^{s}h(Y_{k}^{\delta^n},\delta_{k}^{n})\, dk} f(Y_{s}^{\delta^n},\delta_{s}^{n}) ds  - \ey \int_{0}^{T_{n_{0}}}e^{\int_{0}^{s}h(Y_{k},\delta_{k})\, dk} f(Y_{s},\delta_{s}) ds  \right| < \frac{\varepsilon}{3}.
\]
Fix $n \geq \max\{n_{0},n_{0}'\} $ and consider
\begin{multline*}
\left|\ey \int_{0}^{T_{n}}e^{\int_{0}^{s}h(Y_{k}^{\delta^n},\delta_{k}^{n})\, dk} f(Y_{s}^{\delta^n},\delta_{s}^{n}) ds  - \ey \int_{0}^{+\infty}e^{\int_{0}^{s}h(Y_{k},\delta_{k})\, dk} f(Y_{s},\delta_{s})ds\right| \\ \leq \left|\ey \int_{0}^{T_{n_{0}}}e^{\int_{0}^{s}h(Y_{k}^{\delta^{n}},\delta_{k}^{n})\, dk} f(Y_{k}^{\delta^n},\delta_{s}^{n}) ds  -\ey \int_{0}^{T_{n_{0}}}e^{\int_{0}^{s}h(Y_{k},\delta_{k})\, dk} f(Y_{s},\delta_{s}) ds  \right| \\ + \ey \int_{T_{n_0}}^{T_{n}}e^{\int_{0}^{s}h(Y_{k}^{\delta^n},\delta_{k}^{n})\, dk} |f(Y_{s}^{\delta^n},\delta_{s}^{n})| ds + \ey \int_{T_{n_{0}}}^{+\infty}e^{\int_{0}^{s}h(Y_{k},\delta_{k})\, dk} |f(Y_{s},\delta_{s})| ds   < \varepsilon.
 \end{multline*}
\end{proof}
The above result suggest that in many problems we can find optimal feedback controls $\delta^{n}(y,t)$ for a sequence of finite time horizon formulations  and prove that they are convergent to the infinite horizon feedback control $\delta(y,t)$. To complete this reasoning  we should also check wether $\lim_{n \to \infty}Y_{t}^{\delta^{n}} =Y_{t}^{\delta}$. It will be useful to use the following result.

\begin{prop}
Let $b_{k}(y,t)$, $k \in \mathbb{N}$ be a family of continuous functions such that there exist a constant $K>0$ and a sequence $K_{n}>0$, $n \in \mathbb{N}$ (independent of $k$) such that 
\begin{gather*}
|b_{k}(y,t) - b_{k}(\bar{y},t)|\leq K_{n}|y-\bar{y}|, \quad y,\bar{y} \in B(0,n), \; t\in [0,T], \\
|b_{k}(y,t)|  \leq K(1+|y|), \quad y \in \mathbb{R}^{N},  t\in [0,T].
\end{gather*} 
Suppose further that $\lim_{k \to \infty}b_{k}(y,t)= b(y,t)$, where $b(y,t)$  is a continuous function and  $Y^{k}$, $k \in \mathbb{N}$ is the sequence of solutions to 
\[
dY^{k}_{t}=b_{k}(Y^{k}_{t},t)dt+dW_{t}.
\]
Then $\lim_{k \to \infty} Y^{k}_{t}=Y_{t}$ a.s. for all $t>0$ , where $Y$ is the solution to 
\[
dY_{t}=b(Y_{t},t)dt+dW_{t}.
\]
\end{prop}
\begin{proof}
Assume first that $K_{n}=K'$ for all $n \in \mathbb{N}$. Then,
\begin{multline*}
|Y^{k}_{t}-Y_{t}| \leq \iot |b_{k}(Y^{k}_s,s)-b(Y_s,s)| ds \leq  \iot |b_{k}(Y^{k}_s,s)-b_k(Y_s,s)| ds +  \iot |b_{k}(Y_s,s)-b(Y_s,s)| ds \\
\leq K' \iot |Y^{k}_{s}-Y_{s}| ds + \iot |b_{k}(Y_s,s)-b(Y_s,s)| ds.
\end{multline*}
Using the Gronwall inequality and the fact that $\lim_{k \to \infty}\iot |b_{k}(Y_s,s)-b(Y_s,s)| ds=0$  for all $t>0$, we obtain that  $\lim_{k \to \infty}Y^{k}_{t}=Y_{t}$ a.s. for any $t>0$.
Let's consider now the general problem and define the sequence 
$$
b_{k}^{n}(y,t)=\begin{cases}
b_{k}(y,t) \quad  & \text{if} \quad |y| \leq n,\\
 b_{k}(y,t)\left(2-\frac{|y|}{n}\right), & \text{if} \quad n \leq |y| \leq 2n,\\
  0        &\text{if} \quad  |y| \geq 2n.  \end{cases}
$$ 

Fix $n \in \mathbb{N}$ and consider the sequence of diffusions
\[
dY^{k,n}_{t}=b_{k}^{n}(Y^{k,n}_{t},t)dt+dW_{t}.
\]

We have already proved that 
 $\lim_{k \to \infty}Y^{k,n}_{t}=Y_{t}^{n}$. Since $b_{k}^{n}(y,t)= b_{k}(y,t)$ for all $y \in B(0,n)$, then by Friedman \cite[Chapter 5, Theorem 2.1]{friedman2} 
 \[
 P(\sup_{0 \leq t \leq \tau_{n,k}}|Y_{t}^{k,n}-Y_{t}^{k}|=0)=1,\; P(\sup_{0 \leq t \leq \tau_{n}}|Y^{n}_{t}-Y_{t}|=0)=1
 \]
for all $k,n \in \mathbb{N}$, where $\tau_{k,n}$ is the first exit of the process $Y^{k}$ from $B(0,n)$. Using \eqref{inmax} we know that   
\[
|Y_{t}^{k,n}|\leq \left(L_{1}T+ |y|+\max_{t \in [0,T]}|W_{t}| \right)e^{KT}
\]
Let's define $\Omega_{N}=\{ \omega \in \Omega | \left(L_{1}T+ |y|+\max_{t \in [0,T]}|W_{t}| \right)e^{KT} \leq N\}$. It is important to note that $\bigcup_{N=1}^{+\infty}\Omega_{N}= \Omega$.

If we fix $N \in \mathbb{N}$ and  take any $\omega \in \Omega_{N}$ then $|Y_{t}^{k,n}(\omega)| \leq N$ for all $k,n \in \mathbb{N}$. But we know that $\sup_{0 \leq t \leq T}|Y_{t}^{N}(\omega)-Y_{t}(\omega)|=0$ and $\sup_{0 \leq t \leq T}|Y_{t}^{k,N}(\omega)-Y_{t}^{k}(\omega)|=0 $ for  almost all $\omega \in \Omega_{N}$. 
Applying that reasoning to each $N \in \mathbb{N}$ we obtain that $P(\quad \text{for all} \quad t \in [0,T] \; \lim_{k \to \infty} Y^{k}_{t}=Y_{t})=1.$

\end{proof}
\begin{rem} \label{robust} All results presented so far can be extended  to minimax  problems with Hamilton Jacobi Bellman Isaacs equations of the form
\begin{equation*} 
u_{t}+ \frac{1}{2} \Delta u +\max_{\delta \in D} \min_{\eta \in \Gamma} \biggl([i(y,\delta)+l(\delta,\eta)] \nabla u  +  h(y,\delta,\eta) u+ f(y,\delta,\eta)\biggr)=0, \quad t\in [0,T), y \in \mathbb{R}^{N},
\end{equation*}
together with its infinite horizon analogue. In that case it is possible to derive stochastic representation of the Kalton-Elliott form:
\begin{align*} 
u(y,t)=  \sup_{\delta \in \mathcal{D}}  \inf_{\eta \in \mathcal{N}} \mathbb{E}_{y,t}^{l(\delta,\eta(\delta))} \biggl( \int_{t}^{T}e^{\int_{t}^{s}h(Y_{k},\delta_{k},\eta(\delta_{k}))\, dk} f(Y_{s},\delta_{s},\eta(\delta_{s})) ds + e^{\int_{t}^{T}h(Y_{k},\delta_{k},\eta(\delta_{k}))\, dk}g(Y_{T})\biggr),
\end{align*} 
where $Y$ is a solution  to $dY_{t}=i(Y_{t},\delta_{t},\eta(\delta_{t}))dt +  dW_{t}$, $\mathcal{D}$ is the class of all progressively measurable processes taking values in $D$, $\mathcal{N}$ is the family of all functions: $\eta:D \times [0,+\infty) \times \Omega \to \Gamma$ with the property that for all $\delta \in \mathcal{D}$ the process $(\eta(\delta_{t}):=\eta(\delta_{t},t,\cdot)|\; 0 \leq t < +\infty)$ is progressively measurable. The expression $\mathbb{E}_{y,t}^{l(\delta,\eta(\delta))}$ is used to denote that the expectation is taken under the measure $Q^{l(\delta,\eta(\delta))}$, where 
\[
\frac{dQ^{l(\delta,\eta(\delta))}}{dP}= \exp{\left[\int_{0}^{T}l(\delta_{s},\eta(\delta_{s}))dW_{s} - \frac{1}{2} \int_{0}^{T}|l(\delta_{s},\eta(\delta_{s}))|^{2} ds \right]}.
\]
For more details see Zawisza \cite[Lemma 4.1]{zawisza}. Once we establish such representation we are able to repeat all results and proofs contained in this paper. For other tractable minimax problems and possible stochastic representation we recommend the work of Fleming and Hernandez \cite{fleher}. 
\end{rem} 

\section{Optimal consumption - investment problem}

 \subsection{Consumption-investment problem}
Suppose that the investor can invest in two primitive securities: a
bank account $(B_{t}, 0 \leq t < + \infty)$ and a share $(S_{t},0 \leq t <
+ \infty)$. We assume also that prices are affected by additional observable stochastic factor  $(Y_{t},0 \leq t < + \infty)$. This factor can represent an additional source of an  uncertainty such as: a stochastic volatility, a stochastic interest rate or other  economic conditions. Our economy is given by the following system of  stochastic differential equations
 \begin{equation*} \label{model}
\begin{cases}
dB_{t} &=r(Y_{t}) B_{t} dt,  \\
dS_{t} &=[r(Y_{t})+b(Y_{t})] S_{t}  dt + \sigma(Y_{t}) S_{t}  dW_{t}^{1},   \\
dY_{t} &=i(Y_{t}) dt + (\rho dW_{t}^{1} + \sqrt{1-\rho^{2}} dW_{t}^{2}),
\end{cases}
\end{equation*}
where $W^{1}$ and $W^{2}$ are independent Wiener processes and  $\rho$ is a correlation coefficient. 
The dynamics of the investors wealth process $(X^{\pi,c}_{t}, 0 \leq t < + \infty )$ is given by the stochastic differential equation
\begin{equation} \label{wealth}
\begin{cases}
d X_{t} = (r(Y_{t})X_{t} + \pi_t b(Y_{t})X_{t}) dt +\pi_{t}  \sigma(Y_{t}) X_{t}  dW_{t}^{1}-c_{t}X_{t}dt,\\
 X_{s}=x,
 \end{cases}
 \end{equation}
where $x$ denotes a current wealth of the investor, $\pi$ we can interpret as a capital invested in $S_{t}$, whereas $c$ is a consumption rate. We assume that $\pi$ and $c$ are progressively measurable and are allowed to take values only in intervals $[-R,R]$ and $[0,m]$ respectively. 
The objective of the investor is to maximize
\[
\mathcal{J}^{\pi,c}(x,y,t)=\et \left[e^{-wT} \frac{(X_{T}^{\pi})^{\gamma}}{\gamma} + \itT e^{-ws} \frac{(c_{s} X_{s})^{\gamma}}{\gamma} ds  \right],  
\]
or its infinite horizon analogue
\[
\mathcal{K}^{\pi,c}(x,y)=\ex \left[  \int_{0}^{+\infty} e^{-ws} \frac{(c_{s} X_{s})^{\gamma}}{\gamma} ds  \right],
\]
where $w>0$ is a discount factor. Aforementioned models are some extensions of models propose, for instance, in: Chang et. al \cite{chang}, Korn and Kraft \cite{korn}, Trybu\l a \cite{tryb},  however in opposite to those works we restrict here on the case when $\pi$, $c$ are constrained processes. Unconstrained problem will be treated elsewhere.
Note that under some mild conditions on the process $\pi$ and $c$, there exists a unique strong solution to equation \eqref{wealth} and is given by
\[
X_{t}=x e^{\int_{s}^{t}[  r(Y_{k}) +b(Y_{k})\pi_{k} - \frac{1}{2}\sigma^{2}(Y_{k})\pi_{k}^{2}-c_{k} ]dk + \int_{t}^{T}\sigma(Y_{k})\pi_{k} dW_{k}^{1}}.
\]
The above process is determined under the starting condition $X_{s}=x$.
Therefore,  functions $\mathcal{J}^{\pi,c}(x,y,t)$ and $\mathcal{K}^{\pi,c}(x,y)$ can be transformed in the following way
\begin{multline*}
\mathcal{J}^{\pi,c}(x,y,t)=\frac{x^{\gamma}}{\gamma}\mathbb{E}^{Q^{\pi}_{T}}_{y,0} \Biggr[ e^{\int_{t}^{T}\left(\gamma[r(Y_{s}) +b(Y_{s})\pi_{s} - \frac{1}{2}(1-\gamma)\sigma^{2}(Y_{s})\pi_{s}^{2}- c_{s}]-w \right)ds} \\+ \int_{t}^{T}e^{\int_{t}^{s} \left(\gamma[r(Y_{k}) +b(Y_{k})\pi_{k} - \frac{1}{2}(1-\gamma)\sigma^{2}(Y_{k})\pi_{k}^{2}- c_{k}]-w \right)dk} c_{s}^{\gamma} ds \Biggl],
\end{multline*}
\[
\mathcal{K}^{\pi,c}(x,y) = \lim_{T \to +\infty}\mathbb{E}^{Q^{\pi}_{T}}_{y,0} \int_{0}^{T}e^{\int_{0}^{s}\left(\gamma[r(Y_{k}) +b(Y_{k})\pi_{k} - \frac{1}{2}(1-\gamma)\sigma^{2}(Y_{k})\pi_{k}^{2}- c_{k}]-w\right) dk} c_{s}^{\gamma} ds
\]
where 
\[
\frac{dQ^{\pi}_{T}}{dP}=e^{ - \frac{1}{2} \int_{0}^{T}\sigma^{2}(Y_{k})\pi_{k}^{2}dk + \int_{0}^{T}\sigma(Y_{k})\pi_{k} dW_{k}^{1}} 
\]
and $\pi_{k}=0$ for all $k \leq t$ (for the finite horizon case).
That shows that  the term $\frac{x^{\gamma}}{\gamma}$ can be omitted, and it is worth to consider only the function which is dependent only on $(y,t)$.  
The Girsanow Theorem  gives us the motivation to consider HJB of the form  

\begin{multline}
\label{eqportfolio}
  u_{t} + \frac{1}{2}u_{yy}+ i(y)u_{y} +\max_{\pi \in [-R,R]} \left( \rho \pi \sigma(y)  u_{y}+\left[\gamma b(y)\pi-\frac{1}{2} (\gamma - \gamma^{2})  \pi^{2} \sigma^{2}(y)\right]u   \right) \\ +  \max_{c \in [0,m]} \left(-\gamma cu+c^{\gamma}\right)+ [\gamma r(y)-w]u =0,
\end{multline}
with terminal condition $u(y,T)=1$.
Having Assumption \ref{as} and Assumption \ref{as2} in mind,  we assume the following.

\begin{as} \label{as4}  There exist  $L_{1}>0$, $L_{2} < 0$  that for $i$, $\zeta=b,\sigma^{2},r,i$ and all $y,\bar{y} \in \mathbb{R}^{N}$, we have
\begin{align} 
|\zeta(y)-\zeta(\bar{y})|\leq L_1 |y- \bar{y}|, \notag \\ 
(y-\bar{y})[i(y) -i(\bar{y})]  \leq L_{2}|y-\bar{y}|^{2} \notag .
\end{align}
\end{as}
\begin{prop}

Suppose that all conditions of Assumption \ref{as4}  are  satisfied. Then 
 there exists  $\ctwo$ solution to \eqref{eqportfolio}. Moreover, any Borel measurable maximizer in
 \eqref{eqportfolio} is an optimal feedback strategy for $\mathcal{J}^{\pi,c}(x,y,t)$.
\end{prop}

\begin{proof}
The existence of a smooth solution  to  \eqref{eqportfolio} was proved in Theorem \eqref{unbounded}. Let $u$ stands for the solution  constructed in that theorem. To prove that maximizer in  \eqref{eqportfolio} is an optimal feedback strategy it is sufficient to observe that for any strategy $\pi$
\[
  \mathbb{E}_{y,t}^{Q_{T}^{\pi}} \sup_{0 \leq s \leq T}e^{\int_{t}^{s}h(Y_{k},\delta_{k})\, dk} |u(Y_{s},s)| < +\infty
\]
and we can apply standard verfication argument (see for example the reasoning of Zawisza \cite[Appendix,Theorem 6.1]{Zawisza2}).
\end{proof} 

\begin{as} \label{as5} There exists a deterministic function $\kappa(t,n)$, $t>0$, $n \in \mathbb{N}$,  continuous in $t$ that for any progressively measurable control $(\pi)$ taking values in $[-R,R]$, we have 
\begin{equation*} 
\mathbb{E}_{y,0}^{Q^{\pi}_{t}} e^{\int_{0}^{t}( h(Y_{k},\pi_{k}))\, dk} \leq \kappa(t,n), \quad y \in B(0,n), \quad \int_{0}^{+ \infty} \kappa(t,n) dt < + \infty, 
\end{equation*}
where $h(y,\pi)= \gamma[r(y) +b(y)\pi - \frac{1}{2}(1-\gamma)\sigma^{2}(y)\pi^{2}] - w$.
\end{as}

Now it is right time to consider infinite horizon HJB:
\begin{multline}
\label{eqportfolio2}
  \frac{1}{2}u_{yy}+ i(y)u_{y} +\max_{\pi \in [-R,R]} \left( \rho \pi \sigma(y)  u_{y}+\left[\gamma b(y)\pi-\frac{1}{2} (\gamma - \gamma^{2})  \pi^{2} \sigma^{2}(y)\right]u   \right) \\ +  \max_{c \in [0,m]} \left(-\gamma cu+c^{\gamma}\right)+ [\gamma r(y)-w]u =0.
\end{multline}
\begin{prop}
Suppose that all conditions of Assumption \ref{as4} and Assumption \ref{as5} are  satisfied. Then 
 there exists  $\mathcal{C}^{2}(\mathbb{R}^{N})$ solution to \eqref{eqportfolio2}. Moreover, any Borel measurable maximizer in
 \eqref{eqportfolio2} is an optimal feedback strategy for $\mathcal{K}^{\pi,c}(x,y)$.
\end{prop}

\begin{proof}
A smooth classical solution to  \eqref{eqportfolio2} was constructed in Theorem \ref{main}. To prove that the maximizer in that equation determines the optimal strategy it is sufficient to prove analogue to Proposition \ref{convergence}.
\end{proof}

The lemma below shows how to determine the discount factor $w$ to be sure that all conditions of Assumption \ref{as5} are satisfied.
\begin{lem} Suppose that there exist $\alpha,\beta, P, Q \geq 0$, $\alpha \neq 0$  that $$i(y)\leq -\alpha y + \beta, \quad \gamma r(y) - w \leq -P + Qy, \quad \delta \in D, \; \eta \in \Gamma, \; y  \in \mathbb{R}.$$
 Then  for all continuous processes  $\pi$, we have
 \begin{multline*}
\int_{0}^{t}\left(\gamma[r(Y_{k}) +b(Y_{k})\pi_{k} - \frac{1}{2}(1-\gamma)\sigma^{2}(Y_{k})\pi_{k}^{2}] - w \right)dk\\
 \leq  \left[ \frac{ \gamma Q\beta}{\alpha} -\gamma P\right] t + \frac{\gamma Qy^{+}}{\alpha}+ \int_{0}^{t}\psi(Y_{s},s,t)ds + \gamma Q \int_{0}^{t} \frac{1}{\alpha} \left(1-e^{\alpha(s-t)}\right) dW_{s}^{\pi} ,
\end{multline*}
where
 $$\psi(y,s,t)=\max_{\pi \in [-R,R]} \left[\left[\frac{\gamma Q\rho}{\alpha}(1-e^{\alpha(s-t)}) +\gamma b(y)\right]\pi \sigma(y)dk  - \frac{1}{2}(\gamma-\gamma^{2})\sigma^{2}(y)\pi^{2} - w \right]$$
and $W^{\pi}:=\rho W^{1,\pi} + \sqrt{1-\rho^{2}}W^{2}$, $W^{1,\pi}_{t}=W^{1}_{t}-\int_{0}^{t} \pi_{s} \sigma(Y_{s}) ds $.
\end{lem}
\begin{proof}
Note that under the measure $Q_{T}^{\pi}$ the process $Y$ has the following dynamics
$$dY_{t}=[i(Y_{t}) + \rho \pi_{t}\sigma(Y_{t})]dt +  dW_{t}^{\pi}.$$
Repeating the steps from the proof of Proposition \ref{example} we get

\begin{equation*}
 \int_{0}^{t} Y_{s} \leq \ \frac{y^{+}}{\alpha}  + \frac{\beta}{\alpha} t + \rho \int_{0}^{t}  e^{-\alpha s}\int_{0}^{s} e^{\alpha k} \pi_{k}\sigma(Y_{k})dk ds + \int_{0}^{t} \frac{1}{\alpha} \left(1-e^{\alpha(s-t)}\right) dW_{s}^{\pi}.
\end{equation*}
The integration by parts gives us
\[
\int_{0}^{t}  e^{-\alpha s}\int_{0}^{s} e^{\alpha k} \rho \pi_{k}\sigma(Y_{k}) dk ds = \frac{\rho}{\alpha} \int_{0}^{t} (1-e^{\alpha(k-t)})\pi_{k} \sigma(Y_{k})dk.
\]
Therefore
\[
 \int_{0}^{t}\left(\gamma r(Y_{k}) - w\right)dk \leq -Pt + \frac{Qy^{+}}{\alpha}+\frac{Q\rho}{\alpha} \int_{0}^{t} (1-e^{\alpha(k-t)})\pi_{k}\sigma(Y_{k}) dk  + \frac{Q\beta}{\alpha} t +  Q \int_{0}^{t} \frac{1}{\alpha} \left(1-e^{\alpha(s-t)}\right) dW_{s}^{\pi}
\]
and consequently
\begin{multline*}
  \int_{0}^{t}\gamma[r(Y_{k}) +b(Y_{k})\pi_{k} - \frac{1}{2}(1-\gamma)\sigma^{2}(Y_{k})\pi_{k}^{2}] - w dk \\ \leq  \left[ \frac{ \gamma Q\beta}{\alpha} -\gamma P\right] t + \frac{\gamma Qy^{+}}{\alpha}+ \int_{0}^{t}\left[\frac{\gamma Q\rho}{\alpha}(1-e^{\alpha(k-t)}) +\gamma b(Y_{k})\right]\pi_{k} \sigma(Y_{k})dk \\ - \frac{1}{2}\int_{0}^{t}\left[(\gamma-\gamma^{2})\sigma^{2}(Y_{k})\pi_{k}^{2} - w \right]dk + \gamma Q \int_{0}^{t} \frac{1}{\alpha} \left(1-e^{\alpha(s-t)}\right) dW_{s}^{\pi}.
 \end{multline*}
\end{proof}

Above results can be easily extended to multiasset and multifactor models. For possible direction of generalization see for instance, Berdjane and Pergamenshchikov \cite{perga}, Noh and Kim \cite{nohkim}. Moreover Remark \ref{robust} indicates possible extensions in robust portfolio optimization problems (see Schied \cite{schied}, Flor and Larsen \cite{flor} or Gagliardini et al.  \cite{gagliardini}).

\end{document}